\documentclass[a4paper,reqno,11pt]{article}
\usepackage[a4paper, margin=1in]{geometry}
\usepackage[american]{babel}
\usepackage{upref}
\usepackage{varioref}
\usepackage{enumerate}
\usepackage[ansinew]{inputenc}
\usepackage{mathtools}
\usepackage{amssymb}
\usepackage{amsmath}
\usepackage{amsthm}
\usepackage[colorlinks,
  pdftitle={A Solution to the 1-2-3 Conjecture},
  pdfview=FitH]{hyperref}
\allowdisplaybreaks[2]

\newtheorem{theorem}{Theorem}
\newtheorem{definition}[theorem]{Definition}
\newtheorem{lemma}[theorem]{Lemma}

\newcommand{\NN}{\mathbb N}
\newcommand{\ZZ}{\mathbb Z}

\renewcommand{\epsilon}{\varepsilon}

\begin{document}
\title{A Solution to the 1-2-3 Conjecture}
\author{Ralph Keusch \\ \small{ralphkeusch@gmail.com}}
\date{\today}
\maketitle
\begin{abstract}
We show that for every graph without isolated edge, the edges can be assigned weights from $\{1,2,3\}$ so that no two neighbors receive the same sum of incident edge weights. This solves a conjecture of Karo\'{n}ski, {\L}uczak, and Thomason from 2004.
\end{abstract}

\section{Introduction}\label{section:introduction}
Let $G=(V,E)$ be a simple graph. A $k$-edge-weighting is a function $\omega: E \rightarrow \{1, \ldots, k\}$. Given an edge-weighting $\omega$, for each vertex $v \in V$ we denote by $s_{\omega}(v):=\sum_{w \in N(v)} \omega(\{v,w\})$ its \emph{weighted degree}. We say that two vertices $v,w \in V$ have a coloring conflict if $s_{\omega}(v)=s_{\omega}(w)$ and $\{v,w\} \in E$. If there is no coloring conflict in the graph, $\omega$ is called \emph{vertex-coloring}. We are interested to find the smallest integer $k$ that admits a vertex-coloring $k$-edge-weighting for the graph $G$. This question arose as the local variant of the graph irregularity strength problem, where one seeks to find a $k$-edge-weighting so that \emph{all} nodes receive different weighted degrees \cite{chartrand1986irregular}. In 2004, Karo\'{n}ski, {\L}uczak, and Thomason conjectured that for each connected graph with at least two edges, a vertex-coloring $3$-edge-weighting exists \cite{karonski2004edge}. Soon after, the problem was referred to as 1-2-3 Conjecture and gained a lot of attention due to its elegant statement.

Karo\'{n}ski et al.\ verified the conjecture for $3$-colorable graphs \cite{karonski2004edge}. Afterwards, Addario-Berry, Dalal, McDiarmid, Reed, and Thomason provided the first finite, general upper bound of $k=30$ \cite{addarioberry2007vertexcolouring}. The general result was improved to $k=16$ by Addario-Berry et al.\ \cite{addarioberry2008degree} and further to $k=13$ by Wang and Yu \cite{wang2008onvertex}. In 2010, Kalkowski, Karo\'{n}ski, and Pfender made a big step and proved upper bounds of $k=6$ and $k=5$, using a simple algorithmic argument \cite{kalkowski2009vertex, kalkowski2010vertexcoloring}.

More results have been dedicated to specific graph classes. For $d$-regular graphs, a bound of $k=4$ has been proven for $d \le 3$ \cite{karonski2004edge}, for $d=5$ \cite{bensmail2018result}, and then in general \cite{przybylo2021theconjecture}. Furthermore, Przyby{\l}o gave an affirmative answer to the conjecture for $d$-regular graphs, given that $d \ge 10^8$ \cite{przybylo2021theconjecture}. In addition, the conjecture was confirmed by Zhong for ultra-dense graphs, i.e., for all graphs $G=(V,E)$ where the minimum degree is at least $0.99985|V|$ \cite{zhong2018theconjecture}. Recently, Przyby{\l}o asserted the statement as well for all graphs where the minimum degree is sufficiently large \cite{przybylo2020conjecture}. Concretely, by applying the Lov\'{a}sz Local Lemma, he proved that there exists a constant $C>0$ such that the conjecture holds for all graphs with $\delta(G)\ge C\log\Delta(G)$.  However, not always $3$ weights are necessary. For instance, a random graph $G(n,p)$ asymptotically almost surely admits a $2$-edge-weighting without coloring conflicts \cite{addarioberry2008degree}. For all $d$-regular bipartite graphs \cite{chang2011vertex}, Chang, Lu, Wu, and Yu have shown that $k=2$ is possible as well, given that $d \ge 3$. Regarding the computational complexity, Dudek and Wajc proved that it is NP-complete to determine whether a given graph $G$ supports a vertex-coloring $2$-edge-weighting \cite{dudek2011complexity}, whereas the same decision problem is in $P$ for bipartite graphs \cite{thomassen2016flow}. 

Many closely related problems have been analyzed. A natural variant are total weightings as introduced by Przyby{\l}o and Wo\'{z}niak \cite{przybylo2010conjecture}, where the edges receive weights from $\{1, \ldots, k\}$ as before, but additionally each vertex gets a weight from $\{1, \ldots, \ell\}$. The weighted degree of a vertex is then defined as the sum of all incident edge weights, plus the weight that it received itself. Przyby{\l}o and Wo\'{z}niak conjectured that for each graph there exists a vertex-coloring total weighting with vertices and edges weighted by the set $\{1,2\}$. While this question is still open, a simple argument from Kalkowski shows that each graph admits a vertex-coloring total weighting with vertex weights from $\{1,2\}$ and edge weights from $\{1,2,3\}$ \cite{kalkowski2009note}.

A weaker version of vertex-coloring edge-weightings can be obtained by defining the vertex colors as multisets of incident edge-weights, instead of sums \cite{karonski2004edge, addarioberry2005vertex}. Recently, Vu\v{c}kovi\'{c} reached the optimal bound $k=3$ for this variant \cite{vuckovic2018multiset}. Vice versa, a harder variant are list colorings, where each edge $e \in E$ has its own list $L(e)$ of allowed edge-weights \cite{bartnicki2009weight}. In fact, the application of Alon's Nullstellensatz \cite{alon1999combinatorial} led to significant results on this intriguing problem \cite{seamone2014bounding, cao2021total, zhu2022every} and on its variant for total weightings \cite{przybylo2011total, wong2011total, wong2014every}. Many more variations of vertex-coloring edge-weightings have been studied, e.g., variations for hypergraphs \cite{kalkowski2017conjecture, bennett2016weak} or directed graphs \cite{bartnicki2009weight, khatirinejad2011digraphs}. For a general overview of the progress on the $1$-$2$-$3$ Conjecture and on related problems, we refer to the early survey of Seamone \cite{seamone2012theconjecture} and to the recent survey of Grytczuk \cite{grytczuk2020from}.

Turning back to the original question, the general upper bound was recently lowered to $k=4$ \cite{keusch2023vertex}. With the present paper, we close the final gap and confirm the conjecture.

\begin{theorem}\label{thm:main}
Let $G=(V,E)$ be a graph without connected components isomorphic to $K_2$. Then there exists an edge-weighting $\omega:E \rightarrow \{1,2,3\}$ such that for each edge $\{v,w\} \in E$,
$$\sum_{u \in N(v)}\omega(\{u,v\}) \neq \sum_{u \in N(w)}\omega(\{u,w\}).$$
\end{theorem}

In Section~\ref{section:mainideas}, we give an overview of the proof strategy, describe how the proof is elaborated upon the ideas from \cite{keusch2023vertex}, and collect several auxiliary results. Afterwards, in Section~\ref{section:main} we formally prove the theorem. Finally, we conclude with a few remarks in Section~\ref{section:remarks}.

\section{Main ideas and proof preparations}\label{section:mainideas}

We use the following notation. Let $G=(V,E)$ be a graph, let $W \subseteq V$, and let $C=(S,T)$ be a cut. Then we denote by $E(W)$ the edge set of the induced subgraph $G[W]$ and by $E(S,T)$ the subset of edges having an endpoint in both $S$ and $T$ (the cut edges of $C$). For a vertex $v \in V$, $N(v)$ stands for its neighborhood and $\deg_W(v) := |N(v) \cap W|$ is the number of neighbors in $W$. Finally, for two disjoint subsets $S, T \subseteq V$, denote by $G(S,T)$ the bipartite subgraph with vertex set $S \cup T$ and edge set $E(S,T)$.

As a starting point, let us summarize the strategy that was introduced in \cite{keusch2023vertex} to construct a vertex-coloring edge-weighting with weights $\{1,2,3,4\}$. There, we started with a maximum cut $C=(S,T)$ and initial weights from $\{2,3\}$, making the weighted degrees of nodes in $S$ even and those of nodes in $T$ odd. Depending on the remaining coloring conflicts, an auxiliary flow problem on $G(S,T)$ was carefully designed. Then, the resulting maximum flow yielded a collection of edge-disjoint paths, along which the edge-weights could be changed in order to make the edge-weighting vertex-coloring.

We are going to extend that approach as follows. We partition the vertex set into two sets $R$ and $B$ of \emph{red} and \emph{blue} nodes, where the red vertices form an independent set. We start by giving each edge weight $2$ and apply the strategy from \cite{keusch2023vertex} only to the subgraph $G[B]$, consequently only taking a maximum cut $C=(S,T)$ of $G[B]$. However, when putting the weights onto $E(B)$, we do not yet finalize the vertex-weights of the blue nodes. Instead, we only ensure that there are no coloring conflicts inside $S$ and inside $T$, which is possible with the weight set $\{1,2,3\}$. Afterwards, we cautiously construct a weighting for $E(R,B)$ such that the weighted degrees of vertices in $R$ remain even, but the weighted degrees of the blue nodes become odd, without creating new conflicts inside $S$ and $T$. More precisely, the weighted degree of nodes in $S$ should obtain values $1 \pmod{4}$ and those of $T$ obtain values $3 \pmod{4}$. Consequently, all coloring conflicts will be resolved and the edge-weighting becomes vertex-coloring.

Unfortunately, our construction requires that the set $B$ is of even cardinality, forcing us to handle several different situations when proving Theorem~\ref{thm:main} in Section~\ref{section:main}. In some cases, the described strategy only works when one or even two vertices are removed from the graph. Afterwards, when re-inserting the nodes, augmenting the edge-weighting to the full graph sometimes requires an additional round of weight modifications, for instance along a path $p$.

We now start with the formal preparations for the proof. To find a suitable independent set $R$ of red nodes, we will apply the following simple result.

\begin{lemma}\label{lemma:divide}
Let $G=(V,E)$ be a connected graph, let $v,w \in V$, and let $p$ be a shortest $v$-$w$-path. Then there exists an independent set $R \subseteq V$ such that
\begin{enumerate}[(i)]
\item the graph $G(R,V \setminus R)$ is connected, and
\item the path $p$ is alternating between $R$ and $V \setminus R$, where we can choose whether $v \in R$ or $v \in V \setminus R$.
\end{enumerate}
\end{lemma}
\begin{proof}
Let $p=\{v_1:=v, v_2, \ldots, v_k:=w\}$ be a shortest $v$-$w$-path in $G$. If $v=v_1$ is required to be in $R$, we start with $R:=\{v_1\}$, otherwise we start with $R:=\emptyset$. Next, we put every second vertex of $p$ into $R$. Because $p$ is a shortest $v$-$w$-path, $R$ remains an independent set and the required properties hold at least for the induced subgraph $G[\{v_1, \ldots, v_k\}]$.

Let $v_{k+1}, \ldots, v_n$ be an ordering of the remaining vertices of $V$ (if there are any), such that for each $i>k$, $v_i$ has at least one neighbor $v_j$ with $j<i$. We are going to proceed the remaining vertices one after another, thereby extending $R$, and prove by induction that for each $i>k$, the graph $G[\{v_1, \ldots, v_{i}\}]$ achieves property (i).

Consider a vertex $v_i$ and assume that $G[\{v_1, \ldots, v_{i-1}\}]$ satisfies the precondition. If $v_i$ already has a neighbor $v_j \in R$, we do not extend $R$, otherwise we extend the set $R$ by adding the node $v_i$. In both cases, $R$ obviously remains an independent set and $G(R, \{v_1, \ldots, v_i\} \setminus R)$ is connected, thus the statement follows by induction.
\end{proof}

Once having partitioned the vertex set into the independent set $R$ of red nodes and the set $B:=V \setminus R$ of blue nodes, we will start assigning weights to the edges $E(B)$. At the same time, we introduce an odd-valued \emph{designated color} $f(v)$ for each blue vertex $v \in B$, so that the function $f$ is a proper vertex-coloring. We thereby take into account that the blue-red-edges contribute to the weighted degrees as well. Because we only have edge-weights $\{1,2,3\}$ available, our capabilities are limited and we have to keep the weighted degrees in $B$ even for the moment. But we can construct the edge-weighting of $E(B)$ so that the current weights and the designated colors almost coincide and only differ by $1$. Later, we will overcome the remaining differences when carefully assigning edge-weights to $E(R,B)$. 

With the following lemma, we adapt the key ideas from \cite{keusch2023vertex} to our setting. We will typically apply it to the subgraph $G[B]$ and to the function $h(v):=2\deg_R(v)$, to obtain an edge-weighting $\omega$ of $E(B)$ and a function of designated colors $f: B \rightarrow \{1,3,5, \ldots\}$. Recall that we denote by $s_{\omega}(v)$ the weighted degree of a vertex $v$ under the weighting $\omega$.

\begin{lemma}\label{lemma:weighting}
Let $G=(V,E)$ be a not necessarily connected graph and let $h:V \rightarrow \{0, 2, 4, \ldots\}$ be a function attaining only even values. Then there exists an edge-weighting $\omega: E \rightarrow \{1,2,3\}$ and a function $f:V \rightarrow \{1, 3, 5, \ldots\}$ attaining only odd values such that 
\begin{enumerate}[(i)]
\item $f(v) \neq f(w)$ for each edge $\{v,w\} \in E$, and
\item $|s_{\omega}(v) + h(v) - f(v)| = 1$ for all $v \in V$.
\end{enumerate}
\end{lemma}

We prove the lemma with the flow-based strategy that was introduced in \cite{keusch2023vertex}. A key step towards the statement is the following auxiliary result.

\begin{lemma}[Lemma~2 in \cite{keusch2023vertex}]\label{lemma:flow}
Let $G=(V,E)$ be a graph, let $C=(S,T)$ be a maximum cut of $G$, let $F\subseteq E(S) \cup E(T)$,  and let $\sigma$ be an orientation of the edge set $F$. Furthermore, let $G_{C, F, \sigma}$ be the auxiliary directed multigraph network constructed as follows.
\begin{enumerate}[(i)]
\item As vertex set, take $V$, and add a source node $s$ and a sink node $t$.
\item For each edge $\{u,v\} \in E(S,T)$, insert the two arcs $(u,v)$ and $(v,u)$, both with capacity $1$.
\item For each edge $\{u,v\} \in F$ with corresponding orientation $(u,v) \in \sigma$, insert arcs $(s, u)$ and $(v, t)$, both with capacity $1$, potentially creating multi-arcs. Do not insert $(u,v)$.
\end{enumerate}
Then in the network $G_{C, F, \sigma}$, there exists an $s$-$t$-flow of value $|F|$.
\end{lemma}

\begin{proof}[Proof of Lemma~\ref{lemma:weighting}]Let $G=(V,E)$ be a graph and let $h:V \rightarrow \{0, 2, 4, \ldots\}$. In a first step, we define designated colors $f(v)$ of odd parity such that two neighbors always receive distinct colors. Afterwards, we construct the edge-weighting $\omega$ that satisfies (ii).

Let $\{v_1, \ldots, v_n\}$ be an arbitrary ordering of $V$ and let $C=(S,T)$ be a maximum cut of $G$. We assign the designated colors $f(v_i)$ to all vertices one after another. We aim to define designated colors such that all $v_i \in S$ receive a color $f(v_i) \equiv 1 \pmod{4}$ and each $v_i \in T$ receives a color $f(v_i) \equiv 3 \pmod{4}$. 

Consider a vertex $v_i \in V$, assume that $v_1, \ldots, v_{i-1}$ already got a designated color, and let $s(v_i) := h(v_i)+2\deg(v_i)$. Denote by $k_i \ge 0$ the number of neighbors $v_j$ with $j<i$ that are on the same side of the cut as $v_i$. To define a suitable value $f(v_i)$, we consider the following set of $2k_i+2$ odd values:
$$S(v_i) := \{s(v_i)-2k_i-1, \ldots, s(v_i)-1, s(v_i)+1, \ldots, s(v_i)+2k_i+1\}.$$

We shall choose a value for $f(v_i)$ from $S(v_i)$. Because the value of $f(v_i) \pmod{4}$ is already determined, half of the elements from $S(v_i)$ are not allowed, so $k_i+1$ potential choices are remaining. At most $k_i$ of them are blocked by values $f(v_j)$ from the neighbors $v_j$ of $v_i$ with a smaller index. Hence, there exists at least one suitable value $f(v_i) \in S(v_i)$ with the following properties:
\begin{itemize}
\item $f(v_i) \neq f(v_j)$ for all neighbors $v_j$ of $v_i$ with $j<i$,
\item $f(v_i) \equiv 1 \pmod{4}$ if $v_i \in S$ and $f(v_i) \equiv 3 \pmod{4}$ if $v_i \in T$, and
\item $|f(v_i)-s(v_i)| \leq 2k_i+1.$
\end{itemize}

Fix a value for $f(v_i)$ from $S(v_i)$ that satisfies these three properties simultaneously. We repeat this procedure for all vertices one after another to achieve property (i) of the statement. It remains to define the edge-weighting $\omega$ that fulfills (ii). 

Let $V^+ := \{v_i \in V: f(v_i) > s(v_i)\}$ and $V^- := \{v_i \in V: f(v_i) < s(v_i)\}$. Moreover, for all $v_i \in V$ let 
$$g(v_i) := \begin{cases}
\frac{1}{2} (f(v_i)-s(v_i)-1), & \text{if } v_i \in V^+, \\
\frac{1}{2} (f(v_i)-s(v_i)+1), & \text{if } v_i \in V^-.
\end{cases}$$
Observe that for all $1 \le i \le n$, $|g(v_i)| \le k_i$ by construction. In order to apply Lemma~\ref{lemma:flow}, we construct a subset $F \subseteq E(S) \cup E(T)$ and an orientation $\sigma$ of $F$ as follows. 

For each vertex $v_i \in S$, choose $|g(v_i)|$ neighbors $v_j \in S$ with smaller index (i.e., $j<i$) and add the $|g(v_i)|$ edges $\{v_i, v_j\}$ to $F$. If $v_i \in V^+$, increase the weight of the $|g(v_i)|$ edges $\{v_i, v_j\}$ to $3$ and add the orientations $(v_i, v_j)$ to $\sigma$. Vice versa, if $v_i \in V^-$, decrease the weight of the edges $\{v_i,v_j\}$ to $1$ and add the orientation $(v_j, v_i)$ to $\sigma$. After having executed the described modifications on the edge weights for all vertices in $S$, the weighted degree of each $v_i \in S$ potentially received changes when considering itself and when considering vertices with higher index, resulting in a current value of
$$t(v_i) := 2\deg(v_i) + 2g(v_i) + |\{w:(w,v_i) \in \sigma\}| - |\{w:(v_i,w) \in \sigma\}|,$$
no matter whether $v_i \in V^+$ or $v_i \in V^-$.

For each vertex $v_i \in T$, choose $|g(v_i)|$ neighbors $v_j \in T$ with $j<i$ and add the $|g(v_i)|$ edges $\{v_i, v_j\}$ to $F$. If $v_i \in V^+$, change the weight of these edges to $3$ and always add the orientation $(v_j, v_i)$ to $\sigma$. If $v_i \in V^-$, always add the orientation $(v_i, v_j)$ to $\sigma$ and change the edge-weights to $1$. Mind the differences compared to $S$ regarding the orientations. Under the modified weighting, the current weighted degree of $v_i \in T$ is
$$t(v_i) := 2\deg(v_i) + 2g(v_i) + |\{w:(v_i,w) \in \sigma\}| - |\{w:(w,v_i) \in \sigma\}|.$$

Having defined $F$ and $\sigma$, we proceed by constructing the auxiliary multigraph $G_{C,F,\sigma}$ as specified in the statement of Lemma~\ref{lemma:flow}. Thereby, each edge of $F$ leads to exactly one arc incident to $s$ and one arc incident to $t$, where $s$ and $t$ are the two additional nodes inserted into the graph. For each node $v_i$, the construction is such that the number of arcs from $s$ to $v_i$ is $|\{w:(v_i,w) \in \sigma\}|$ and the number of arcs from $v_i$ to $t$ is $|\{w:(w,v_i) \in \sigma\}|$. 

We now apply Lemma~\ref{lemma:flow} to $G$ and obtain an $s$-$t$-flow of size $|F|$ in the auxiliary multigraph $G_{C,F,\sigma}$. As all edges have capacity $1$, there are $|F|$ edge-disjoint $s$-$t$-paths in $G_{C,F,\sigma}$. Consider such a directed path $p=(s, u_1, \ldots, u_m, t)$, and let $p' = \{u_1, \ldots, u_m\}$ be its induced, undirected subpath in the bipartite graph $G(S,T)$. Unless $u_1 = u_m$ (which happens when $p'$ is an empty path), we modify the weighting $\omega$ of each edge $\{u_i, u_{i+1}\} \in p'$ as follows: increase the weight to $3$ if $u_i \in S$, and decrease the weight to $1$ if $u_i \in T$. In other words, we alternately increase or decrease the edge weights along the path. The weighted degrees of the internal nodes $u_2, \ldots, u_{m-1}$ thereby do not change, in contrast to those of $u_1$ and $u_m$. The weighted degree of $u_1$ increases by $1$, if $u_1 \in S$, and decreases by $1$, if $u_1 \in T$. Regarding $u_m$, its weighted degree increases by $1$, if $u_m \in T$, and decreases by $1$, if $u_m \in S$. When $u_1=u_m$, there is no change on the weighted degree of this node. We repeat the described modification on $\omega$ for all $|F|$ paths provided by Lemma~\ref{lemma:flow}. Denote by $\omega$ the resulting edge-weighting.

As we found $|F|$ edge-disjoint $s$-$t$-paths in the auxiliary network $G_{C,F,\sigma}$, each arc starting at $s$ and each arc arriving at $t$ is included in exactly one of the $s$-$t$-paths on which we modified edge-weights. Looking at a vertex $v_i \in S$, each of the arcs $(s,v_i)$ is contained in an $s$-$t$-path $p$, so for each arc $(s,v_i)$ we increased the weighted degree of $v_i$ by $1$ when handling $p$. Similarly, for each arc $(v_i,t)$ we decreased the weighted degree of $v_i$ by $1$ when handling the $s$-$t$-path which contained that arc. The only exceptions are $s$-$t$-paths of the form $\{s, v_i, t\}$ or $\{s, v_i, \ldots, v_i, t\}$ where we didn't make any weight modifications. But there, the arcs $(s,v_i)$ and $(v_i,t)$ cancel each other out. Summing up the changes on the weighted degree of $v_i$, we deduce that it holds
$$s_{\omega}(v_i) = t(v_i) + |\{w:(v_i,w) \in \sigma\}| - |\{w:(w,v_i) \in \sigma\}| = 2\deg(v_i) + 2g(v_i).$$ 
Vice versa, consider a vertex $v_i \in T$. By the same argument, for each arc $(s,v_i)$ we decreased the weighted degree of $v_i$ by $1$ and for each arc $(v_i,t)$ we increased the weighted degree of $v_i$ by $1$. Adding up the changes, we again have
$$s_{\omega}(v_i) = t(v_i) + |\{w:(w,v_i) \in \sigma\}| - |\{w:(v_i,w) \in \sigma\}| = 2\deg(v_i) + 2g(v_i).$$ 

Putting everything together and plugging in the definition of $s(v_i)$ and then the definition of $g(v_i)$, we conclude that for each $v_i \in V$ it holds
$$
|s_{\omega}(v_i)+h(v_i)-f(v_i)| =|2\deg(v_i)+2g(v_i)+h(v_i)-f(v_i)| 
= |2g(v_i)+s(v_i)-f(v_i)| 
= 1.
$$
\end{proof}

With Lemma~\ref{lemma:weighting}, we get an edge-weighting for $E(B)$ and designated weights $f(v)$ for the vertices $v \in B$. For each blue node $v$, some amount $\alpha(v)$ of incident edge weights is missing to actually achieve its designated color with its weighted degree. This additional weight $\alpha(v)$ will be gained via the edges between $B$ and $R$. With Lemma~\ref{lemma:edgesbetween} below, we indeed find an edge-weighting for the subgraph $G(R,B)$ where the weighted degree of each $v \in B$ is $\alpha(v)$.

As discussed above, we also have to cover some uncomfortable cases where some vertices will be removed from the graph, whereupon the situation becomes more complex. In some of these situations, there will be a set $R' \subseteq R$ of vertices that should attain a weighted degree of odd parity  (instead of even parity as expected). Moreover, for a few blue vertices $v \in B$, $\alpha(v)$ may be even. In some other cases, the edge-weighting is forced to accomplish some requirements on a fixed path $p$, to guarantee that the weights can be further changed along the edges of $p$ without destroying the entire structure. Lemma~\ref{lemma:edgesbetween} contains several additional properties that will help us to handle the exceptional cases.

\begin{lemma}\label{lemma:edgesbetween}
Let $G=(V,E)$ be a connected bipartite graph with parts $B$ and $R$. Let $\alpha:B \rightarrow \NN$ be a function such that $\alpha(v) \in \{2\deg(v)-1, 2\deg(v), 2\deg(v)+1\}$ for all $v \in B$. Moreover, let $R' \subseteq R$ such that $|R'| + \sum_{v \in B} \alpha(v)$ is even. Then there exists an edge-weighting $\omega: E \rightarrow \{1,2,3\}$ such that
\begin{enumerate}[(i)]
\item $s_{\omega}(v)$ is even for all $v \in R \setminus R'$,
\item $s_{\omega}(v)$ is odd for all $v \in R'$, and
\item $s_{\omega}(v) = \alpha(v)$ for all $v \in B$.
\end{enumerate}
Moreover, let $p=\{v_1, \ldots, v_k\}$ be a fixed path with $k \ge 3$ and $v_1, v_k \in B$. Then $\omega$ satisfies in addition
\begin{enumerate}
\item[(iv)] $\omega(\{v_1,v_2\}) \neq 1$ if $\alpha(v_1)=2\deg(v_1)+1$ and $\omega(\{v_1,v_2\}) \neq 3$ otherwise,
\item[(v)] $\omega(\{v_{i-1},v_i\})+\omega(\{v_{i},v_{i+1}\}) \in \{3,4,5\}$ for each $1 < i < k$ with $v_i \in B$, and  
\item[(vi)] $\omega(\{v_{k-1},v_k\}) \neq 1$ if $\alpha(v_k)=2\deg(v_k)+1$  and $\omega(\{v_{k-1},v_k\}) \neq 3$ otherwise. 
\end{enumerate}
\end{lemma}
\begin{proof}
We start by setting the initial weight of all edges to $2$. Let $T$ be a spanning tree of $G$ that includes the fixed path $p$. We construct $\omega$ by only changing the weights of a subset $E_o$ of $T$-edges. Consider $T$ as a rooted tree with arbitrary root $r$ and denote for each $v \neq r$ by $par(v)$ its parent in the rooted tree.

Let $V_o := R' \cup \{v \in B: \alpha(v) \neq 2\deg(v)\}$ be the set of all vertices that shall obtain an odd weighted degree. Note that the assumption on $R'$ guarantees that $|V_o|$ is even. While constructing the set $E_o$, denote for each $v \in V$ by $E_o(v)$ the subset of edges from $E_o$ that are incident to $v$. We want to arrange the set $E_o$ so that for all nodes $v \in V$, $|E_o(v)|$ is odd if and only if $v \in V_o$. 

We start with the leafs of $T$. For each leaf $\ell$, put $\{\ell, par(\ell)\}$ into $E_o$ if and only if $\ell \in V_o$. We then iterate to the internal nodes of $T$ and repeat the idea: we consider each node $v$ only after having handled all its children, and then decide whether we put $\{v, par(v)\}$ into $E_o$ or not, thereby always ensuring that $|E_o(v)|\equiv 1 \pmod{2}$ if and only if $v \in V_o$. For the root $r$, the argument does not work, since $r$ has no parent node. However, because each edge from $E_o$ contributes to two sets $E_o(v)$, the sum $\sum_{v \in V}|E_o(v)|$ must be even.  Thus,
$$0 \pmod{2} \equiv \sum_{v \in V}|E_o(v)| \pmod{2} \equiv \big(|E_o(r)| + |V_o \setminus \{r\}|\big) \pmod{2},$$ 
and since $|V_o|$ is even, we see that also for the root $r$, the value $|E_o(r)|$ is odd if and only if $r \in V_o$. 

We are now going to modify the weighting of $E_o$. As the graph $G$ is bipartite with parts $R$ and $B$, it is sufficient to only consider sets $E_o(v)$ where $v \in B$. For each $v \in B$, we change the weights of the edges in $E_o(v)$ according to the following rules.

If $\alpha(v) = 2\deg(v)-1$, decrease the weights of $\tfrac{1}{2}(|E_o(v)|+1)$ edges in $E_o(v)$ to $1$, and increase the weights of all other edges in $E_o(v)$ to $3$. Then indeed the weighted degree of $v$ becomes 
$$2(\deg(v)-|E_o(v)|)+\tfrac{1}{2}(|E_o(v)|+1)+\tfrac{3}{2}(|E_o(v)|-1)=2(\deg(v)-|E_o(v)|)+2|E_o(v)|-1=\alpha(v).$$
If $v$ is not a vertex of the fixed path $p$, we can distribute these weight modifications arbitrarily among $E_o(v)$, otherwise there are some restrictions. In situations where $v$ is an internal vertex of $p$ and both incident edges are contained in $E_o(v)$, we ensure that one of the two edges gets weight $1$ and the other $3$. If $v=v_1$ is the starting vertex of $p$ and $\{v_1,v_2\} \in E_o(v_1)$, set $\omega(\{v_1,v_2\})=1$ and distribute the other weights (if there are any) arbitrarily among $E_o(v)$. Similarly, if $v=v_k$ is the ending vertex of $p$ and $\{v_{k-1},v_k\} \in E_o(v_k)$, our only restriction is to put weight $1$ on the edge $\{v_{k-1},v_k\}$.

If $\alpha(v) =2\deg(v)+1$, decrease the weights of $\tfrac{1}{2}(|E_o(v)|-1)$ edges of $E_o(v)$ to $1$, and increase all other weights of edges of $E_o(v)$ to $3$. Then, again it holds $s_{\omega}(v)=\alpha(v)$. Similarly as above, ensure that if $v$ is an internal node of $p$, not both edges on $p$ incident to $v$ get the same odd weight. Moreover, if $v$ is the starting or ending vertex of $p$, ensure that the edge on $p$ incident to $v$ does not receive weight $1$ if it is contained in $E_o(v)$. 

Finally, if $\alpha(v)=2\deg(v)$, $|E_o(v)|$ is even and we assign the weight $1$ to one half of the edges in $E_o(v)$ and weight $3$ to the other half, to assure $s_{\omega}(v)=\alpha(v)$. Once more, if $v$ is an internal vertex of $p$, ensure that not both incident edges on $p$ get the same odd weight. Furthermore, if $v$ is the starting or ending vertex of $p$, again take care of that the edge on $p$ incident to $v$ does not receive weight $3$.

With the described weight modifications, the edge-weighting $\omega$ is defined and (iii)-(vi) are fulfilled. Regarding the red vertices, for each node $v \in R$, $|E_o(v)|$ is odd if and only if $v \in R'$, and exactly the incident edges that are contained in $E_o(v)$ have been weighted with an odd value. Thus, $\omega$ achieves properties (i) and (ii) as well.
\end{proof}

\section{Proof of Theorem~\ref{thm:main}}\label{section:main}

In Section~\ref{section:mainideas}, we prepared the proof with several auxiliary results. The plan is now to define an independent set $R$ with Lemma~\ref{lemma:divide}, to define $B:=V \setminus R$, to use Lemma~\ref{lemma:weighting} for finding an edge-weighting for $G[B]$, and finally to apply Lemma~\ref{lemma:edgesbetween} when extending the edge-weighting to the remaining edges. The crucial point behind this strategy is that the set $B$ is required to have even cardinality, making the proof significantly more technical. 

We therefore describe three basic situations regarding the sets $R$ and $B$, and, in some cases, one additional vertex $v_0 \notin R \cup B$. We demonstrate for each situation that a vertex-coloring edge-weighting with weights $\{1,2,3\}$ can be constructed, always using Lemma~\ref{lemma:weighting} and Lemma~\ref{lemma:edgesbetween} in combination. Afterwards in the actual proof of Theorem~\ref{thm:main}, we will show by a case distinction that for each graph, the problem can actually be reduced to one of the three basic situations.

We use the following definition to annotate a partition $V=R \cup B$ that achieves all required properties.

\begin{definition}\label{definition:partition}
Let $G=(V,E)$ be a connected graph and let $V=R \cup B$ be a partition of the vertex set into two disjoint subsets of red and blue nodes. We say that $(R,B)$ is a good $R$-$B$-partition of $G$ if $R$ is an independent set, the bipartite subgraph $G(R,B)$ is connected, and $|B| \equiv 0 \pmod{2}$.
\end{definition}

Obviously, the ideal situation occurs when there is a good $R$-$B$-partition known for the entire graph. Lemma~\ref{lemma:basicsituation} shows how we then find a suitable edge-weighting.

\begin{lemma}\label{lemma:basicsituation}
Let $G=(V,E)$ be a connected graph and let $R \cup B$ be a good $R$-$B$-partition of $G$. Then there exists an edge weighting $\omega:E \rightarrow \{1,2,3\}$ such that the weighted degrees $s_{\omega}$ yield a proper vertex coloring of $G$ and such that $s_{\omega}(v)$ is even if and only if $v \in R$.
\end{lemma}
\begin{proof}
For all $v \in B$, let $h(v) := 2\deg_R(v)$. We apply Lemma~\ref{lemma:weighting} to $G[B]$ and to $h$, and obtain a weighting $\omega_1$ of $E(B)$ together with a function $f$ on $B$, standing for the designated final weighted degrees of the nodes.

Next, for each $v \in B$ let $\alpha(v) := f(v) - s_{\omega_1}(v)$ be the difference between the designated weighted degree and the already received incident edge weights. By Lemma~\ref{lemma:weighting}~(ii), we know that $\alpha(v) \in \{2\deg_R(v)-1, 2\deg_R(v)+1\}$. Moreover, by putting $R' := \emptyset$ and using the assumption that $|B|$ is even, it follows that the value of $|R'| + \sum_{v \in B} \alpha(v)$ is even.

We apply Lemma~\ref{lemma:edgesbetween} to the bipartite subgraph $G(R,B)$ and to $\alpha$, without considering any path $p$, to obtain a weighting $\omega_2$ for $G(R,B)$ where each vertex $v \in R$ receives an even-valued weighted degree $s_{\omega}(v) := s_{\omega_2}(v)$. Each vertex $v \in B$ gets an additional weight of $\alpha(v)$, hence combining the two weightings $\omega_1$ and $\omega_2$, for $v \in B$ we have $s_{\omega}(v):=s_{\omega_1}(v)+s_{\omega_2}(v)=f(v)$. By Lemma~\ref{lemma:weighting}, $f$ only attains odd values and for any two neighbors $v,w \in B$ we have $f(v) \neq f(w)$. Because $R$ is an independent set, the weighted degrees $s_{\omega}$ indeed properly color the vertices of the graph $G$.
\end{proof}

In the next situation, there exists an additional vertex $v_0$ which is not included in $R$ or $B$, i.e., we only have a good $R$-$B$-partition of $G[V \setminus \{v_0\}]$. Consequently, the weighted degree of $v_0$ must be even and coloring conflicts between $v_0$ and its neighbors in $R$ can arise. To solve these conflicts, we put odd weights on the edges between $v_0$ and $R$. Carefully choosing weights $1$ or $3$, we can ensure that the weighted degree of $v_0$ is different from all of its neighbors. However, the argument only works when $v_0$ has at least $2$ neighbors in $R$. Furthermore, if $\deg_R(v_0)$ is odd, $v_0$ is required to have at least one neighbor in $B$, as the respective edge is needed for making $s_{\omega}(v_0)$ even-valued.

\begin{lemma}\label{lemma:remainingvertex}
Let $G=(V,E)$ be a graph and let $v_0 \in V$ be a vertex such that $G[V \setminus \{v_0\}]$ is connected. Moreover, let $(R,B)$ be a good $R$-$B$-partition of $G[V \setminus \{v_0\}]$ such that $\deg_R(v_0) \ge 2$ and such that either $\deg_R(v_0)$ is even or $\deg_B(v_0) \ge 1$. Then there exists a vertex-coloring edge weighting $\omega:E \rightarrow \{1,2,3\}$ of $G$.
\end{lemma}
\begin{proof}
Let $h(v) := 2|N(v) \setminus B|$ for all $v \in B$. We will construct the weighting $\omega$ in three steps: a weighting $\omega_1$ for $E(B)$, a weighting $\omega_2$ for the edges between $R$ and $B$, and a weighting $\omega_3$ for the edges incident to $v_0$. The final weighting $\omega$ is then the combination of $\omega_1$, $\omega_2$, and $\omega_3$.

We first apply Lemma~\ref{lemma:weighting} to $G[B]$ and $h$ to obtain a weighting $\omega_1$ of the edges set $E(B)$, together with designated final weighted degrees $f(v)$ for the blue nodes. By Lemma~\ref{lemma:weighting}~(ii) and our choice of $h$, for all $v \in B \setminus N(v_0)$ it holds
$$f(v)-s_{\omega_1}(v) = h(v) \pm 1 \in \{2\deg_R(v)+1, 2\deg_R(v)-1\},$$
whereas for all $v \in N(v_0) \cap B$ we have
$$f(v)-s_{\omega_1}(v) = h(v) \pm 1 \in \{2\deg_R(v)+3, 2\deg_R(v)+1\}.$$

Since the edges between $v_0$ and $R$ will receive an odd weight, we put $R' := N(v_0) \cap R$, taking care of that the weighted degrees of these nodes will be even-valued at the end. Next, for all $v \in B \setminus N(v_0)$ we let $\alpha(v) := f(v) - s_{\omega_1}(v)$. Regarding $N(v_0) \cap B$, we distinguish two cases. 
\begin{itemize}
\item If $|R'|$ is even, we put $\alpha(v) := f(v)-s_{\omega_1}(v)-2$ for all $v \in N(v_0) \cap B$. Note that by construction, it holds $\alpha(v) \in \{2\deg_R(v)+1,2\deg_R(v)-1\}$.
\item If $|R'|$ is odd, then by assumption $\deg_B(v_0) > 0$. Fix a vertex $u_0 \in N(v_0) \cap B$ and set $\alpha(u_0) := 2\deg_R(u_0)$. For all $v \in B \cap N(v_0) \setminus \{u_0\}$, let again $\alpha(v) := f(v)-s_{\omega_1}(v)-2$.
\end{itemize}

Because $|B|$ is even, we ensured in both cases that $|R'| + \sum_{v \in B} \alpha(v)$ is even. We apply Lemma~\ref{lemma:edgesbetween} to the graph $G(R,B)$ and to $\alpha$ (but without any path $p$) and obtain a weighting $\omega_2$ for the bipartite graph such that for a vertex $v \in R$, $s_{\omega_2}(v)$ is odd if and only if $v \in R'$. Moreover, for each vertex $v \in B$, it holds $s_{\omega_2}(v)=\alpha(v)$. 

We now introduce the third weighting $\omega_3$ for the edges that are incident to $v_0$. For all $v \in N(v_0) \cap R$ put $\omega_3(\{v,v_0\}) := 1$. If $\deg_R(v_0)$ is odd, we have specified a distinct vertex $u_0 \in N(v_0) \cap B$. Set $\omega_3(\{v_0,u_0\}) := f(u_0)-s_{\omega_1}(u_0)-s_{\omega_2}(u_0)$ and observe that this value is indeed either $1$ or $3$. For all remaining edges $e$ between $v_0$ and $B$, set $\omega_3(e) := 2$. In both of our cases, the value $s_{\omega_3}(v_0)$ thereby becomes even. 

We combine $\omega := \omega_1 + \omega_2 + \omega_3$ to a full edge-weighting of $G$. For $v \in B$ we then have $s_{\omega}(v) = f(v)$, no matter whether $v$ is connected to $v_0$ or not. By Lemma~\ref{lemma:weighting}, $f$ attains only odd values on $B$ and for any two neighbors $v,w \in B$ we have $f(v) \neq f(w)$. For $v \in R$, $s_{\omega_2}(v)$ is odd if and only if $v \in R'$, by Lemma~\ref{lemma:edgesbetween}. However, for all $v \in R'$ we set $s_{\omega_3}(v)=1$, so $s_{\omega}(v)=s_{\omega_2}(v)+s_{\omega_3}(v)$ is even again. Because $s_{\omega}(v_0)=s_{\omega_3}(v_0)$ is even and $R$ is an independent set, it remains to guarantee that there are no coloring conflicts between $v_0$ and its neighbors in the set $R$.

Let $N_R(v_0) := N(v_0) \cap R = \{v_1, \ldots, v_k\}$, where $k \ge 2$ by assumption, and assume w.l.o.g.\ that $s_{\omega}(v_1) \le s_{\omega}(v_2) \le \ldots \le s_{\omega}(v_k)$. If $s_{\omega}(v_i) \neq s_{\omega}(v_0)$ for all $1 \le i \le k$, there are no coloring conflicts and we are done. Otherwise, we increase some edge-weights. Let $x > 0$ be the smallest integer such that $s_{\omega}(v_0)+2x$ is different from all values $s_{\omega}(v_1), \ldots, s_{\omega}(v_k)$, and let $i' \le k$ be maximal such that $s_{\omega}(v_{i'}) < s_{\omega}(v_0) + 2x$. Because at least one $s_{\omega}(v_i)$ is equal to $s_{\omega}(v_0)$, the index $i'$ is well-defined. 

First consider the case $i' \le k-x$. For each $v_i \in N_R(v_0)$ with $i > k-x$, increase the weight of $\{v_i, v_0\}$ from $1$ to $3$ and denote by $\omega'$ the resulting edge-weighting. 
 Then, for $i > k-x \ge i'$, it holds
$$s_{\omega'}(v_i) > s_{\omega}(v_i) > s_{\omega}(v_0) + 2x,$$ 
whereas for $i \le k-x$, we have 
$$s_{\omega'}(v_i)=s_{\omega}(v_i) \neq s_{\omega}(v_0) + 2x.$$
Hence, $s_{\omega'}(v_0)=s_{\omega}(v_0)+2x$ is different from $s_{\omega'}(v_i)$ for all $1 \le i \le k$.

Next, consider the case $i' > k-x$ and $x<k$. We change the weight from $1$ to $3$ for all edges $\{v_0,v_i\}$ where $v_i \in N_R(v_0)$ and $i \ge k-x$. Again, denote by $\omega'$ the resulting edge weighting. For $i \ge k-x$ we then have 
$$s_{\omega'}(v_i)=s_{\omega}(v_i)+2 \neq s_{\omega}(v_0)+2x+2,$$ 
whereas for all $i <k-x$ it holds 
$$s_{\omega'}(v_i)=s_{\omega}(v_i) \le s_{\omega}(v_{i'}) < s_{\omega}(v_0)+2x.$$
Thus, for all $1 \le i \le k$ we achieved $s_{\omega'}(v_0) = s_{\omega}(v_0) + 2x+2 \neq s_{\omega'}(v_i)$.

It remains to consider the case $x=k$. Here, for each $0\le y < k$, the value $s_{\omega}(v_0)+2y$ is attained by one $s_{\omega}(v_i)$. So we have $s_{\omega}(v_i)=s_{\omega}(v_0)+2i-2$ for all $1 \le i \le k$. We only increase the weight of $\{v_0,v_2\}$ to $3$. For the new edge-weighting $\omega'$, it holds $s_{\omega'}(v_0)=s_{\omega}(v_0)+2$, but $s_{\omega'}(v_1)=s_{\omega}(v_0)$, $s_{\omega'}(v_2)=s_{\omega}(v_2)+2=s_{\omega}(v_0)+4$, and, for all $i \ge 3$, 
$$s_{\omega'}(v_i)=s_{\omega}(v_i)=s_{\omega}(v_0)+2i-2 \ge s_{\omega}(v_0)+4.$$
\end{proof}

The last of our three base situations is again given by a vertex $v_0$ which is not contained in $B \cup R$. In contrast to the setting of Lemma~\ref{lemma:remainingvertex}, this time we have $\deg_R(v_0)=1$. Here, a coloring-conflict can only appear between $v_0$ and its single neighbor $u_0 \in R$. If this conflict occurs, we repair it by changing the weights along a cycle that includes $v_0$, so that the weighted degree of $v_0$ changes but that of $u_0$ remains the same.

\begin{lemma}\label{lemma:rearranging}
Let $G=(V,E)$ be a graph and let $v_0 \in V$ be a vertex such that the induced subgraph $G[V \setminus \{v_0\}]$ is connected. Let $(R,B)$ be a good $R$-$B$-partition of $G[V \setminus \{v_0\}]$ and let $u_0 \in R$ such that $N(v_0) \cap R = \{u_0\}$. Suppose that there exists a non-trivial path in $G(R,B)$ that starts and ends in $N(v_0) \cap B$ and does not include $u_0$. Then there exists an edge weighting $\omega:E \rightarrow \{1,2,3\}$ such that the weighted degrees $s_{\omega}$ yield a proper vertex coloring of $G$.
\end{lemma}
\begin{proof}
For all $v \in B$, put $h(v) := 2|N(v) \setminus B|$. We apply Lemma~\ref{lemma:weighting} to $G[B]$ and to $h$, and receive a weighting $\omega_1$ of $E(B)$, together with a function $f$ on $B$ where $f(v)-s_{\omega_1}(v)=h(v)\pm 1$ holds for all $v \in B$. At the end, for each $v \in B$ its weighted degree $s_{\omega}(v)$ shall coincide with $f(v)$.

Let $k \ge 3$ and let $p=\{v_1,v_2, \ldots, v_k\}$ be a path in $G(R,B)$ which does not have $u_0$ as internal node and whose starting and endpoint vertices $v_1, v_k$ are in $N(v_0) \cap B$. Next, for all $v \in N(v_0) \cap B$ define $\beta(v) := f(v)-s_{\omega_1}(v)-2\deg_R(v)$. Since $v_1$ and $v_k$ are connected to $v_0$, we have $\beta(v_1),\beta(v_k) \in \{1,3\}$. In the next step, we define a subset $R' \subseteq R$ and a weighting $\omega_2$ for the edges incident to $v_0$, thereby considering two cases.
\begin{enumerate}[(a)]
\item If $\beta(v_1)=\beta(v_k)$, let $R':=\emptyset$ and set $\omega_2(e)=2$ for each edge $e$ incident to $v_0$. 
\item If $\beta(v_1)\neq \beta(v_k)$, assume w.l.o.g.\ that $\beta(v_1)=3$ and $\beta(v_k)=1$. Set $\omega_2(\{v_0,u_0\}):=3$ and $\omega_2(\{v_0,v_1\}):=3$. For all other edges $e$ incident to $v_0$ (including in particular $\{v_0,v_k\}$), put $\omega_2(e):=2$. Finally, let $R' := \{u_0\}$.
\end{enumerate}

Moreover, for all $v \in B \cap N(v_0)$ we define $\alpha(v) := f(v)-s_{\omega_1}(v)-s_{\omega_2}(v)$, whereas for all $v \in B \setminus N(v_0)$ we set $\alpha(v) := f(v)-s_{\omega_1}(v)$. Note that for all blue vertices we have $\alpha(v) \in \{2\deg_R(v)+1,2\deg_R(v)-1\}$, except for $v_1$ in case (b) where $\alpha(v_1)=2\deg_R(v_1)$. We claim that $|R'| + \sum_{v \in B} \alpha(v)$ is even in both cases (a) and (b). Indeed, in case~(a), this follows because $\alpha(v)$ is odd for all $v \in B$, $|B|$ is even, and $R'$ is empty. In case~(b), it is true because $|B|$ is still even, $|R'|$ is odd, but $s_{\omega_2}(v_1)$ is odd and thus $\alpha(v_1)$ is even. We apply Lemma~\ref{lemma:edgesbetween} to $G(R,B)$, $\alpha$, and $p$, and obtain a weighting $\omega_3$ for the bipartite graph with various properties. Having $\omega_3$ on hand, let $\omega$ be the edge-weighting that combines $\omega_1$, $\omega_2$, and $\omega_3$. 

For each vertex $v \in R$, $s_{\omega_3}(v)$ is odd if and only if $v \in R'$, implying that for all $v \in R \setminus N(v_0)$, $s_{\omega}(v)=s_{\omega_3}(v)$ is even. Moreover, for the only vertex $u_0$ in $R \cap N(v_0)$, $s_{\omega}(u_0) = s_{\omega_2}(u_0)+s_{\omega_3}(u_0)$ is even as well, in both cases (a) and (b). Therefore, $s_{\omega}$ indeed attains even values on the red vertices. Regarding the blue vertices, we have $s_{\omega_3}(v)=\alpha(v)$ for all $v \in B$ and thus $s_{\omega}(v) = s_{\omega_1}(v)+s_{\omega_2}(v)+s_{\omega_3}(v)=f(v)$. By Lemma~\ref{lemma:weighting}, $f$ only attains odd values on $B$ and for any two neighbors $v,w \in B$ we have $f(v) \neq f(w)$. All together, under $\omega$ a coloring conflict can only arise between $v_0$ and $u_0$.

If $s_{\omega}(v_0) \neq s_{\omega}(u_0)$, there is nothing more to do and $s_{\omega}$ properly colors the vertices of $G$. So assume that the two values are equal. Our goal is to create a modified edge-weighting $\omega'$ without coloring conflicts, by only changing the weights on $p$, on $\{v_0, v_1\}$, and on $\{v_0, v_k\}$. We start with the edges that are incident to $v_1$ or $v_k$. There are three sub-cases.

\begin{itemize}
\item If $\beta(v_1)=\beta(v_k)=1$, we decrease the weights of $\{v_0,v_1\}$ and $\{v_0,v_k\}$ by $1$, so we put $\omega'(\{v_0,v_1)\}):=\omega'(\{v_0,v_k\}) := 1$. However, we also have $\alpha(v_1)=2\deg_R(v_1)-1$ and $\alpha(v_k)=2\deg_R(v_k)-1$, implying $\omega(\{v_1,v_2\}) \neq 3$ and $\omega(\{v_{k-1},v_k\}) \neq 3$ by properties (iv) and (vi) of Lemma~\ref{lemma:edgesbetween}. Hence we can \emph{increase} the weights of those two edges by $1$ in order to achieve that $s_{\omega'}(v_1)=s_{\omega}(v_1)$ and $s_{\omega'}(v_k)=s_{\omega}(v_k)$. At the same time, we have $s_{\omega'}(v_0) = s_{\omega}(v_0)-2$.
\item If $\beta(v_1)=\beta(v_k)=3$, we put $\omega'(\{v_0,v_1)\}):=\omega'(\{v_0,v_k\}) := 3$. Observe that here, it holds $\alpha(v_1)=2\deg_R(v_1)+1$ and $\alpha(v_k)=2\deg_R(v_k)+1$. By Lemma~\ref{lemma:edgesbetween}~(iv) and (vi), the weights of $\{v_1,v_2\}$ and $\{v_{k-1},v_k\}$ are not $1$, hence we can \emph{decrease} the weights of the two edges by $1$ to keep the weighted degrees of $v_1$ and $v_k$ the same, whereas $s_{\omega'}(v_0) = s_{\omega}(v_0)+2$.
\item If $\beta(v_1)\neq \beta(v_k)$, we assumed w.l.o.g.\ that $\beta(v_1)=3$ and $\beta(v_k)=1$. Recall that in this situation we put $\omega_2(\{v_0,v_1\})=3$ and $\omega_2(\{v_0,v_k\})=2$, implying $\alpha(v_1)=2\deg_R(v_1)$ and $\alpha(v_k)=2\deg_R(v_k)-1$. Then Lemma~\ref{lemma:edgesbetween} yields $\omega(\{v_1,v_2\}) \neq 3$ and $\omega(\{v_{k-1},v_k\} \neq 3$. We now \emph{decrease} the weights of $\{v_0,v_1\}$ and $\{v_0,v_k\}$ both by $1$ and \emph{increase} the weights of $\{v_1,v_2\}$ and $\{v_{k-1},v_k\}$ by $1$. Again, the weighted degrees of $v_1$ and $v_k$ remain the same, while $s_{\omega'}(v_0) = s_{\omega}(v_0)-2$.
\end{itemize}

In all three cases we achieved $s_{\omega'}(v_0) = s_{\omega}(v_0) \pm 2$. It remains to consider the internal nodes of $p$. So far, we changed the weighted degrees of $v_2$ and $v_{k-1}$ by $+1$ or $-1$ (or, if $k=3$, we  changed the weighted degree of $v_2=v_{k+1}$ by $+2$ or $-2$). For each internal node $v_i \in B$ of $p$ (if there are any), it holds $\omega(\{v_{i-1},v_i\})+\omega(\{v_i,v_{i+1}\}) \in \{3,4,5\}$ by Lemma~\ref{lemma:edgesbetween}~(v). Therefore, there is always a choice to modify the weights of both $\{v_{i-1},v_i\}$ and $\{v_i,v_{i+1}\}$ by $1$ while the weighted degree of $v_i$ remains the same. After these modifications, it holds $|\omega'(e)-\omega(e)|=1$ for each edge $e \in p$. Because $R$ is an independent set, the weighted degree of each internal node $v \in R$ of $p$ remains even and we did not create any new coloring conflicts. Since $u_0 \notin p$, it holds in particular $s_{\omega'}(u_1)=s_{\omega}(u_1)$, and we conclude that the coloring conflict between $v_0$ and $u_0$ has been solved by changing the edge weights along the cycle $\{v_0,v_1, v_2, \ldots, v_k,v_0\}$.
\end{proof}

Before starting with the main proof, we need one last minor lemma.

\begin{lemma}\label{lemma:blocks}
Let $G=(V,E)$ be a connected graph with minimum degree at least $2$. Then there exist two vertices $x,y \in V$ such that $\{x,y\} \in E$ and $G[V \setminus \{x,y\}]$ is connected.
\end{lemma}
\begin{proof}
We do a depth-first-search on $G$ with arbitrary starting vertex $r$ and consider a leaf $x$ on the resulting search tree $T$ with largest depth, i.e., with largest distance on $T$ to the root $r$ among all nodes. Let $y$ be the parent of $x$ in $T$. We claim that $G[V \setminus \{x, y\}]$ is connected. 

Clearly, removing $x$ does not disconnect $G$. All other children of $y$ (if there are any) are leafs of $T$ as well, due to our choice of $x$. As they all have degree at least $2$, they are all connected to at least one other node in $G$. However, because we consider a depth-first search tree, all neighbors of children of $y$ must be ancestors of $y$ in $T$. Hence, we can also remove $y$ without destroying the connectivity of the remaining graph.
\end{proof}

\begin{proof}[Proof of Theorem~\ref{thm:main}]
Assume w.l.o.g.\ that $G$ is connected and contains at least three vertices. The plan is to show that it is always possible to find a good $R$-$B$-partition of either the entire graph $G$ or a large portion of it. Then the good $R$-$B$-partition should allow us to apply either Lemma~\ref{lemma:basicsituation}, Lemma~\ref{lemma:remainingvertex}, or Lemma~\ref{lemma:rearranging} to find a vertex-coloring edge-weighting. However, the fact that $|B|$ needs to be even  requires a careful preparation, leading to a subtle case distinction.

We first assume that there exists $x \in V$ with $\deg(x)=1$. Let $y$ be its only neighbor. Define $V' := V \setminus \{x\}$ and $G' := G[V']$. By Lemma~\ref{lemma:divide} applied to $G'$ and to the trivial path $p:=\{y\}$, there exists an independent set $R$ such that $y \in R$ and $G(R,V \setminus R)$ is connected. Let $B := V' \setminus R$. We distinguish two sub-cases. 

\begin{itemize}
\item If $|B|$ is odd, we can add $x$ to the set $B$ and make $|B|$ even. Afterwards, $(R,B)$ is a good $R$-$B$-partition of $G$, so we can apply Lemma~\ref{lemma:basicsituation} directly to find a vertex-coloring edge-weighting $\omega$ of $G$.
\item If $|B|$ is even, then $(R,B)$ is a good $R$-$B$-partition of $G'$. We apply Lemma~\ref{lemma:basicsituation} to $G'$ and receive a vertex-coloring edge-weighting $\omega$ of $G'$ where $s_{\omega}(y)$ is even. To extend $\omega$ to the full edge set $E$, we put $\omega(\{x,y\}):=2$ . Clearly, the weighted degree of $y$ remains even. As $R$ is an independent set, the only potential coloring conflict is between $x$ and $y$. But $\deg(x)=1$ and $\deg(y)>1$, hence $s_{\omega}(x)=2$, $s_{\omega}(y) > 2$, and $s_{\omega}$ is a proper vertex coloring of the full graph $G$.
\end{itemize}
 
For the remainder, we can assume that the minimum degree of $G$ is at least $2$. By Lemma~\ref{lemma:blocks}, there are two vertices $x,y \in V$ sharing an edge such that removing them does not disconnect the graph. We define $V'' := V \setminus \{x,y\}$ and $G'' := G[V'']$.

We first study the special situation where $\deg(x)=\deg(y)=2$. Let $z_x \in N(x) \setminus \{y\}$ and $z_y \in N(y) \setminus \{x\}$ be the two other neighbors of $x$ and $y$, where $z_x=z_y$ is possible. We apply Lemma~\ref{lemma:divide} to $G''$ and to the trivial path $p := \{z_x\}$, thereby requiring $z_x \in R$, and get an independent set $R$ such that $G(R,V \setminus R)$ is connected. We set $B := V'' \setminus R$ and distinguish four different sub-cases.

\begin{itemize}
\item If $|B|$ is even and $z_y \in R$, we add both $x$ and $y$ to $B$. $|B|$ remains even, $(B,R)$ is a good $R$-$B$-partition, and with Lemma~\ref{lemma:basicsituation}, we find an edge-weighting $\omega$ without coloring conflicts. 
\item If $|B|$ is even and $z_y \in B$ (and thus $z_x \neq z_y$), we put $y$ into $R$, so both neighbors of $x$ are in $R$ while $R$ remains an independent set. By Lemma~\ref{lemma:remainingvertex} applied to $G$ and $v_0 := x$, there exists a vertex-coloring $\omega$ without coloring conflicts.
\item If $|B|$ is odd and $z_y \in B$, we put $x$ into $B$ and $y$ into $R$ to obtain a good $R$-$B$-partition of $G$. Then we are done with Lemma~\ref{lemma:basicsituation}.
\item If $|B|$ is odd but $z_y \in R$, we create an auxiliary graph $\tilde{G}=(\tilde{V},\tilde{E})$ by adding an additional vertex $w$ and setting $\tilde{V} := V'' \cup \{w\}$ and $\tilde{E} := E(V'') \cup \{\{w,z_x\}\}$. After adding $w$ to $B$, $|B|$ becomes even. We apply Lemma~\ref{lemma:basicsituation} to $\tilde{G}$ and get a vertex-coloring edge-weighting $\tilde{\omega}$ for $\tilde{G}$ where $s_{\tilde{\omega}}(v)$ is even if and only if $v \in R$. Note that $\tilde{\omega}(\{w,z_x\})$ must be odd since $w \in B$.

Going back to the original graph $G$, we set $\omega(e) := \tilde{\omega}(e)$ for all edges $e \in E(V'')$. Moreover, let $\omega(\{y,z_y\}) := 2$ and $\omega(\{x,z_x\}) := 1$, making $s_{\omega}(z_x)$ and $s_{\omega}(z_y)$ even and thus keeping $V''$ vertex-coloring. Finally, if $s_{\omega}(z_x)=2$, we put $\omega(\{x,y\}):=3$, in all other cases we set $\omega(\{x,y\}):=1$. Then $s_{\omega}(x)$ is even, $s_{\omega}(y)$ is odd, and $s_{\omega}(x) \neq s_{\omega}(z_x)$, so $s_{\omega}$ is indeed vertex-coloring.
\end{itemize}

It remains to handle the situations where $\deg(x)+\deg(y) > 4$. We assume w.l.o.g.\ that $\deg(x)>2$ and again define specific vertices $z_x$ and $z_y$, this time together with a $z_x$-$z_y$-path $p$. Let $z_x \in N(x) \setminus \{y\}$ and $z_y \in N(y) \setminus \{x\}$, but now we require in addition that either $z_x = z_y$ (then $p$ would be the trivial path), or $\{x,z_y\} \notin E$, $\{y,z_x\} \notin E$, and there exists a shortest $z_x$-$z_y$-path $p$ in $G''$ without internal node that is connected to $x$ or $y$. The path $p$ is only needed for solving a few exceptional sub-cases. To obtain the independent set $R$, we apply Lemma~\ref{lemma:divide} to $G''$ and to the path $p$, this time requiring $z_x \notin R$. Having received the independent set $R$, we set $B := V'' \setminus R$, thus $z_x \in B$. Moreover, by Lemma~\ref{lemma:divide}~(ii), the path $p$ is alternating between $B$ and $R$. 

We start with the cases where $|B|$ is odd. First, we assume that at most one vertex of $x$ and $y$ has neighbors in $R$.

\begin{itemize}
\item If $\deg_R(x)=0$, no matter whether $\deg_R(y)=0$ or not, we put $x$ into $R$ and $y$ into $B$, consequently making $|B|$ even and keeping $R$ an independent set. Moreover, the graph $G(R,B)$ is connected thanks to the edge $\{x,y\}$. Hence, we have a good $R$-$B$-partition of $G$ and with Lemma~\ref{lemma:basicsituation} we directly find a suitable edge-weighting $\omega$.
\item If $\deg_R(y)=0$ and $\deg_R(x) \neq 0$, we put $y$ into $R$ and $x$ into $B$. Again, we find a suitable edge-weighting $\omega$ with Lemma~\ref{lemma:basicsituation}.
\end{itemize}
Next, suppose that $|B|$ is odd, $\deg_R(x) \ge 1$, $\deg_R(y) \ge 1$, and $\deg_R(x)+\deg_R(y) \ge 3$.
\begin{itemize}
\item If $\deg_R(x) \ge 2$, put $y$ into $B$. Then the bipartite subgraph $G(R,B)$ is connected and $\deg_B(x) \ge 1$, so we can apply Lemma~\ref{lemma:remainingvertex} to $G$ and to $v_0:=x$, no matter whether $\deg_R(x)$ is even or odd, and find a suitable edge-weighting $\omega$.
\item If $\deg_R(y) \ge 2$, the same argument as above applies by exchanging $x$ and $y$. Note that here, the assumption $z_x \in B$ does not have any impact.
\end{itemize}
We now assume that $|B|$ is odd, $\deg_R(x)=\deg_R(y)=1$, and $x$ and $y$ actually have a joint neighbor $q$ in $R$.
\begin{itemize}
\item If $N(x)\cap N(y) \cap R = \{q\}$, we create an auxiliary graph $\tilde{G}=(\tilde{V},\tilde{E})$ as follows. Let $\tilde{V} := V \cup \{w\}$, where $w$ is an additional vertex, and set $\tilde{E} := E \setminus \{\{x,y\},\{x,q\}\} \cup \{\{w,y\}\}$. Furthermore, put $x$ and $w$ into $R$ and $y$ into $B$. Because we removed the edge $\{x,q\}$, the set $R$ is an independent set. Moreover, thanks to $y$, $|B|$ becomes even, thus $(R,B)$ is a good $R$-$B$-partition of $\tilde{G}$. We apply Lemma~\ref{lemma:basicsituation} to $\tilde{G}$ and receive a vertex-coloring edge-weighting $\tilde{\omega}$ such that $s_{\tilde{\omega}}$ attains odd values exactly on $B$. Because $\deg_{\tilde{G}}(w)=1$, we have $\tilde{\omega}(\{w,y\})=2$. We now construct an edge-weighting $\omega$ for the original graph $G$ as follows. 

For each edge $e \in E \cap \tilde{E}$, we let $\omega(e)=\tilde{\omega}(e)$. On the two edges $\{x,q\}$ and $\{x,y\}$, we put weight $2$. We observe that $s_{\omega}(y)$ has now the same value in $G$ as $s_{\omega'}(y)$ had in $\tilde{G}$. Moreover, $s_{\omega}(q)$ is still even, hence the only potential coloring conflict is between $x$ and $q$. If $s_{\omega}(x) \neq s_{\omega}(q)$, we are done. Otherwise, let $\mu := \omega(\{y,q\})$.

If $\mu = 2$, we can modify $\omega$ by resetting $\omega(\{y,q\}):=3$, $\omega(\{x,y\}):=1$, and $\omega(\{x,q\}):=1$. The effect is that the weighted degrees of $y$ and $q$ remained the same whereas $s_{\omega}(x)$ decreased by $2$, so we resolved the coloring conflict between $q$ and $x$.

If $\mu$ is odd, we adapt $\omega$ by resetting $\omega(\{x,y\}):=\mu$, $\omega(\{x,q\}):=\mu$, and $\omega(\{y,q\}):=2$. Thereby, the weighted degree of $x$ changes by $\pm 2$ whereas all other weighted degrees remain the same. Thus, in both situations, the coloring conflict is resolved.
\end{itemize}
Under the assumption that $|B|$ is odd, it remains to handle the situations where $x$ and $y$ both have exactly one neighbor in $R$ but not a joint neighbor therein. Denote by $z_x'$ the neighbor of $x$ in $R$. Recall that in $G''$, there exists a $B$-$R$-alternating path $p$ from $z_x$ to $z_y$ without internal vertices that are connected to $x$ or $y$. 

\begin{itemize}
\item If $z_y \in R$, put $y$ into $B$ and observe that $(R,B)$ is a good $R$-$B$-partition of $G[V \setminus \{x\}]$. Let $p' := p \cup \{y,z_y\}$. Then the preconditions of Lemma~\ref{lemma:rearranging} are satisfied with $v_0:=x$, $u_0:=z_x'$, and $p'$, hence we find a vertex-coloring edge-weighting $\omega$.

\item If $z_y \in B$, let $T$ be a spanning tree of $G(R,B)$ containing $p$. By Lemma~\ref{lemma:divide}~(i), $G(R,B)$ is connected, so indeed this spanning tree exists. Let $z_y' \neq z_x'$ be the neighbor of $y$ in $R$. By construction, the path $p$ does not contain $z_x'$ or $z_y'$ as internal node. Therefore, on $T$ there exists a $z_y$-$z_x'$-path without $z_y'$ as internal vertex, or there exists a $z_x$-$z_y'$-path without $z_x'$ as internal vertex. Assume w.l.o.g.\ that we are in the first situation and denote by $p'$ our $z_y$-$z_x'$-path. Define $p'' := p' \cup \{x,z_x'\}$, put $x$ into $B$, and observe that by setting $v_0 := y$ and $u_0 := z_y'$, with the path $p''$ all preconditions of Lemma~\ref{lemma:rearranging} are fulfilled. Again, there exists an edge-weighting $\omega$ that properly colors the vertices of $G$.
\end{itemize}

We now turn to cases where $|B|$ is even. For the remainder of the proof, we don't need the path $p$ anymore, but we still use the vertex $z_x \in N(x) \cap B$. We start with the sub-cases where $\deg_R(x) \ge 1$.

\begin{itemize}
\item If $\deg_R(y) \ge 1$, we simply put $x$ and $y$ to $B$ to get a good $R$-$B$-partition of $G$. We directly find $\omega$ with Lemma~\ref{lemma:basicsituation} applied to $G$.
\item If $\deg_R(y) = 0$, we add $y$ to $R$ and achieve $\deg_R(x) \ge 2$. Thanks to $z_x$, $\deg_B(x)$ is at least $1$, allowing us to use Lemma~\ref{lemma:remainingvertex} with $v_0:=x$ no matter whether $\deg_R(x)$ is even or odd. Again, we find again a vertex-coloring edge-weighting $\omega$.
\end{itemize}
Finally, we are left with the situations where $|B|$ is even and the set $N(x) \cap R$ is empty. Due to the assumption that $\deg(x)>2$, we have $\deg_B(x) \ge 2$. We distinguish the following four sub-cases. 

\begin{itemize}
\item If $\deg_R(y)=0$, we add $y$ to $R$. Let $z_x' \in N(x) \setminus \{y,z_x\}$. Since $G(R,B)$ is connected, there exists a $z_x$-$z_x'$-path $p'$ in $G''$ that alternates between $B$ and $R$. We apply Lemma~\ref{lemma:rearranging} to $G$, to the fixed path $p'$, to $v_0:=x$, and to $u_0:=y$ and see that there exists an edge-weighting $\omega$ without coloring conflicts.
\item If $y$ has both neighbors in $R$ and in $B$, we add $x$ to $R$. Then $\deg_R(y) \ge 2$ and $\deg_B(y) \ge 1$, so we can apply Lemma~\ref{lemma:remainingvertex} with $v_0:=y$, no matter whether $\deg_R(y)$ is even or odd, and find a vertex-coloring edge-weighting $\omega$.
\item If $\deg_B(y)=0$ and $\deg(y)$ is even, we put $x$ into $R$. Observe that the value of $\deg_R(y)$ must be even in this sub-case. Thus, by Lemma~\ref{lemma:remainingvertex} applied to $v_0:=y$, there is an edge-weighting $\omega$ with the desired properties.
\item At last, if $\deg_B(y)=0$ and $\deg(y)$ is odd (and thus at least $3$), we have to be careful. Note that here, we have $N(x) \cap N(y) = \emptyset$. We create an auxiliary graph $\tilde{G}$ as follows. Let $\tilde{E} := E \setminus \{\{x,y\},\{y,z_y\}\} \cup \{x,z_y\}$ and let $\tilde{G} := (V,\tilde{E})$. We keep $R$ as it is, but we add both $x$ and $y$ to $B$. Observe that $\tilde{G}$ is connected, because the degree of $y$ was sufficiently large. Moreover, $(R,B)$ is a good $R$-$B$-partition of $\tilde{G}$ because $x \in B$ is connected to $R$ thanks to the edge $\{x,z_y\}$. With Lemma~\ref{lemma:basicsituation} we then obtain a vertex-coloring edge-weighting $\tilde{\omega}$ for $\tilde{G}$. Note that since $y \in B$, $s_{\tilde{\omega}}(y)$ is odd, so there exists at least one edge $\tilde{e}$ incident to $y$ where $\tilde{\omega}(\tilde{e})$ is odd. This edge $\tilde{e}$ will be helpful below.

Let $\mu := \tilde{\omega}(\{x,z_y\})$. We use $\tilde{\omega}$ to construct a weighting $\omega$ for $G$ as follows. We set $\omega(\{x,y\}) := \omega(\{y,z_y\}) := \mu$, whereas for all edges $e \in E \cap E'$, we keep the edge-weight from $\tilde{\omega}$. Consequently, for all $v \in V \setminus \{y\}$ we have $s_{\omega}(v)=s_{\tilde{\omega}}(v)$, whereas $s_{\omega}(y)=s_{\tilde{\omega}}(y)+2\mu$ remains odd. Because $N(y) \cap B = \{x\}$, the only potential coloring conflict is between the two nodes $x$ and $y$. However, if $x$ and $y$ attain the same weighted degree, we can replace the weight of edge $\tilde{e}$ by the other possible odd value. Then the conflict vanishes, and since the other endpoint of $\tilde{e}$ lies in the independent set $R$ of even-weighted nodes, we did not create any new conflicts.
\end{itemize}

No matter how the structure of the graph is at $x$ and $y$, we have seen by an exhaustive case distinction that it is always possible to reduce the situation to one of the three basic situations that have been covered by Lemma~\ref{lemma:basicsituation}, Lemma~\ref{lemma:remainingvertex}, and Lemma~\ref{lemma:rearranging}. Hence, for each connected graph $G=(V,E)$ with $|V| \ge 3$, the presented approach leads to a vertex-coloring edge-weighting $\omega$ only with weights from $\{1,2,3\}$.
\end{proof}

\section{Concluding remarks}\label{section:remarks}

With this paper, we present a constructive solution to the 1-2-3 Conjecture that is built on top of the ideas from \cite{keusch2023vertex} and uses the independent set $R$ of red nodes as additional key ingredient. We mentioned in Section~\ref{section:introduction} that the decision problem whether there exists a vertex-coloring edge-weighting only from $\{1,2\}$ is NP-hard \cite{dudek2011complexity}. In our proof, the critical step in terms of algorithmic complexity is to find a maximum cut $C=(S,T)$ in the proof of Lemma~\ref{lemma:weighting}. Finding a maximum cut is NP-complete, but our strategy can be adjusted as follows. 

Instead of starting directly with a maximum cut, we could begin with an arbitrary cut. Then we find either a sufficiently large flow, or a strictly larger cut, as demonstrated in the proof of Lemma~\ref{lemma:flow} in \cite{keusch2023vertex}. By repeating the argument, at some point we get a cut $C'=(S',T')$ that leads to a suitable flow, because the size of the cut can not become arbitrarily large. This should actually happen in polynomial time. We therefore believe that using the ideas of our proof, a polynomial-time algorithm for the construction of a suitable edge-weighting can be derived.

Throughout the paper, it was assumed that for each edge the weights are chosen from the set $\{1,2,3\}$. We think that all arguments work as long as the weight set has the form $\{a, a+1, a+2\}$, with $a \in \{1, 3, 5, \ldots\}$. However, it is reasonable to believe that the statement is true for all weight sets of the form $\{a, a+d, a+2d\}$ with $a \in \ZZ$ and $b \in \NN$, but such a generalization does not follow directly from the present proof.

Finally, it is natural to transfer the described technique to total weightings. However, it seems that additional ideas are needed to solve the conjecture of Przyby{\l}o and Wo\'{z}niak \cite{przybylo2010conjecture}. We leave this to future work.

\bibliographystyle{abbrv}
\bibliography{references}

\end{document}